\documentclass{amsart}
\usepackage{amscd, amssymb, verbatim,a4wide,color}
\newcommand{\n}{\noindent}
\usepackage[latin1]{inputenc}
\newcommand{\lr}{\longrightarrow}
\theoremstyle{plain}
\newtheorem{definition}{Definition}[section]
\newtheorem{lemma}[definition]{Lemma}
\newtheorem{thm}[definition]{Theorem}

\newtheorem{cor}[definition]{Corollary}
\newtheorem{rem}[definition]{Remark}

\def\cplxi{{\mskip2mu\sf{i}\mskip2mu}}
\def\Id{{\mathbf{I}}}
\newcommand{\weg}[1]{}
\newcommand{\om}[1]{\stackrel{#1}{\kappa}}
\begin{document}
\title{Some Remarks on Nijenhuis Bracket, Formality, and K\"ahler Manifolds}
\author{Paolo de Bartolomeis and Vladimir S.
Matveev}
\date{\today}
\maketitle
\section{introduction}
A differentiable manifold $M$\, is said to be {\em formal} if the algebra $\wedge^{*}(M)$\, of the differential forms on $M$\, is quasi isomorphic to its DeRham cohomology. (We recall that a morphism between Differential Graded Algebras is said to be a {\em quasi isomorphism} if it induces an isomorphism in cohomology and that two DGA's are said to be {\em quasi isomorphic} if they are equivalent with respect to the equivalence relation generated by quasi isomorphisms (cf. also \cite{PdB1}\,).)\\
It is well known that (cf. \cite{DGMS}\,)

\noindent\fbox{\begin{minipage}{.3\textwidth}$(\wedge^{*}(M),\,d)\,\,\, {\rm is\,\,formal}$\end{minipage}}
$\Longrightarrow$
\fbox{\begin{minipage}{.6\textwidth} {from $H^{*}(M,{\mathbb{R}})$\, we can reconstruct}\\ 
(via its minimal model, Postnikov towers etc...)\\ 
the whole rational \\ (i.e. all the cofinite) homotopy theory of $M$\,.\end{minipage}}

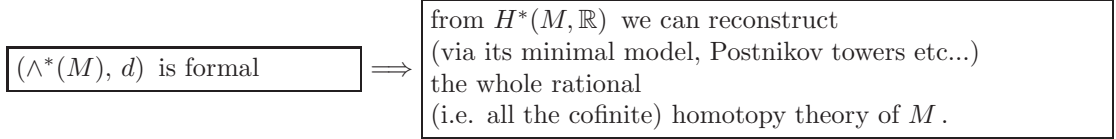
\n One { (actually, almost the only effective)  way to get formality is to be able to produce a suitable derivation  $\delta$ on $\wedge^{*}(M)$\,,   $\delta:\wedge^k \to \wedge ^{k+1}$ (for $k=0,...,n$), satisfying $\delta^{2}=0$\,and such that 
{\em $d\delta$-lemma} holds, i.e.
$(Ker\,d\cap Ker\,\delta)\cap (Im\,d+  Im\,\delta)=Im\,d\delta.$
 }

More precisely, the  following general statements holds: 
\begin{thm} [cf. \cite{DGMS}]
Let $M$ be a smooth manifold with a derivation 
$\delta:\wedge^k \to \wedge ^{k+1}$ (for $k=0,...,n$), satisfying $\delta^{2}=0$ such that  $d\delta$-lemma holds. Then
$$H(\wedge^{*}(M),\,d)=(Ker\,{ d}\cap Ker\,\delta)/Im\,d\delta$$
$$H(\wedge^{*}(M),\,\delta)=(Ker\,d\cap Ker\,\delta)/Im\,d\delta$$
and so $(\wedge^{*}(M),\,d)$\, and $(\wedge^{*}(M),\,\delta)$\, are {\em formal}.
\end{thm}

An example of such a situation is provided by Kähler manifolds: in this case, $\delta=d^{c}:= J^{-1} d J$, where $J$ is the complex structure  (cf. again \cite{DGMS})\,.

We first show (Lemma \ref{L:B}, Remark \ref{R:B})
that the  derivation  $\delta$\, satisfying properties above 
 must be of the form $\delta=d_{R}:=RdR^{-1}$\,, with { $R\in End(TM)$\, (i.e., $R$ is a field of non degenerate 
 linear transformations of the tangent spaces)}.

 { Then, we prove (Lemma \ref{nijen}) that the supercommutation of $d$\, and $\delta=d_{R}$\, (which is a natural, essentially necessary condition to get a $d\delta$-lemma)\, amounts to $N_{R}\equiv 0$\,, $N_{R}$\, being the Nijenhuis tensor of $R$\,. Then,  we are looking for 
 sufficient conditions that ensure the $dd_{R}$-lemma holds. 
 For $R$ self adjoint with respect to a Riemannian metric,  it is done in Section \ref{riem}. For $R$ compatible with an almost symplectic structure this is done in Section \ref{scew}.  
Finally, we show that, if ${}^{t}\!R=-R$\, and $det\,R\equiv1$\,, then
$$N_{R}\equiv 0\,\,\Longrightarrow\,\,N_{J}\equiv 0$$
where $J$\,  is the orthogonal component of $R$\,,  in its polar decomposition and this also provides a new characterization of K\"ahler structures. } 

\section{Preliminary remarks: on the space of derivations}
\n We begin with the following
\begin{lemma}\label{L:B}
Let $M$\, be a smooth compact manifold of dimension $n$\, and let\\ 
$\delta\in End(\wedge^{*}(M))$\,such that:  
\begin{itemize}
\item[a.] $\delta(\alpha\wedge\beta)=\delta\alpha\wedge\beta+(-1)^{|\alpha|}\alpha\wedge{ \delta}\beta$\, and $|\delta|=1$\,, i.e. $\delta\,:\,\wedge^{p}(M)\,\lr\,\wedge^{p+1}(M)$
 \item[b.] $Ker\,\delta\cap\wedge^{0}(M)={\mathbb{R}}$
\item[c.] $\delta^{2}=0$\,;
\end{itemize}
then $R\,:\,X\,\mapsto\,\delta_{X}$\, belongs to $End(TM)$\, and $\delta=d_{R}:=RdR^{-1}$\, (where, clearly, $\delta_{X}\,:\,f\,\mapsto\,\delta f(X)$\,) 
\end{lemma}

\begin{proof} { The linearity of $R$ is evident. By (b), $R$  is nondegenerate.} 
{ In order to prove $\delta= d_R$, let us note that} for any $f\in\wedge^{0}(M),\, X\in TM$\,, we have:
$$d_{R}f(X)=df(RX)=RXf=\delta_{X}f=(\delta f)(X)$$
i.e. $\delta$\, coincides with $d_{R}$\, on $\wedge^{0}(M)$\,; this, together with {(a)}, (c), is sufficient  to insure  $\delta\equiv d_{R}$\,.
\end{proof}

\begin{rem} \label{R:B} 
Assume $\delta$\, satisfies (a), (c) of lemma \ref{L:B}. Suppose 
\begin{itemize} 
\item the $d\delta$-lemma holds, i.e.:
$$(Ker\,d\cap Ker\,\delta)\cap (Im\,d+  Im\,\delta)=Im\,d\delta$$
\item $d\delta+\delta d=0.$
\end{itemize}
Then also (b) of lemma \eqref{L:B} is fulfilled.
\end{rem}
\begin{proof} {  Indeed, } if $f\in Ker\,\delta\cap\wedge^{0}(M)\,\,,\,\,f\ne const$\,, then
$$0\ne df\in (Ker\,d\cap Ker\,\delta)\cap (Im\,d+ Im\,\delta)\,,$$
{ contradicting}   
$df\notin Im\,d\delta$.  \end{proof} 

\vspace{0.3cm}

\n For any $S\in End(TM)$\,, we define the {\em Nijenhuis tensor}\, of $S$\, as
the element 
$$N_{S}\in\wedge^{2}(M)\otimes TM$$
given by
$$N_{S}(X,\,Y):=[SX,\,SY]+S^{2}[X,\,Y]-S[SX,\,Y]-S[X,\,SY]\,;$$

It is known  (and follows direct from definitions),  that 

\begin{itemize}
\item $N_{\Id+S}=N_{S}$ \ (where $\Id:TM\to TM$ is the identity) 

\item for any $\lambda\in C^{\infty}(M,\,{\mathbb{R}})$\,, $N_{\lambda \Id}\equiv 0$

\item if $R\in End(TM)$\, then
$$N_{R^{-1}}(X,\,Y)=R^{-2}N_{R}(R^{-1}X,\,R^{-1}Y)\,.$$
\end{itemize}

\vspace{0.3cm}

\n Let $V$\, be a vector space. For any 
 $L\in End(V)$ we consider 
$$\tau(L)\in End(\wedge V^{*})$$ defined as follows:
\begin{equation} \label{tau} (\tau(L)(\alpha))(v_{1},\,...\,,\,v_{p}):=\sum_{h=1}^{p}\alpha(v_{1},\,...\,,L(v_{h}),\,...\,,\,v_{p})\,.\end{equation}
We recall that a {\em Differential Graded Lie Algebra} (DGLA) is a graded vector space
$${\mathfrak{g}}=\bigoplus_{j\in{\mathbb{Z}}}{\mathfrak{g}}_{j}$$
together with a bilinear map $[\,,\,]\,:\,{\mathfrak{g}}\times{\mathfrak{g}}\,\lr\,{\mathfrak{g}}$\, and a degree one graded derivation $d$\, on ${\mathfrak{g}}$\, in such a way that:
\begin{itemize}
\item $[{\mathfrak{g}}_{j},\,{\mathfrak{g}}_{k}]\subset {\mathfrak{g}}_{j+k}$
\item for homogeneous elements $a,\,b,\,c$\,, we have:
$$
\begin{array}{rcc} [a,\,b]=& -(-1)^{|a||b|}[b,\,a] &  \cr
[a,\,[ b,\,c ] ]= & [ [a,\,b],\,c]+(-1)^{|a||b|}[b,\,[a,\,c] ] & \ \  \mbox{ { Jacobi identity}}\\
d[a,\,b]=&[da,\,b]+(-1)^{|a|}[a,\,db].&  \end{array}$$
\item $d^{2}=0$
\end{itemize}

\vspace{0.3cm}

For example, there is a natural structure of $DGLA$\, on $End(\wedge^{*}(M))$:\\  the { gradation is obvious: $|P|= |P\alpha| - |\alpha|$ },\\ 
and the bracket  $[\ , \, ]$ and the derivation  (we use the letter  $\daleth$ for it) are given by
\begin{itemize}
\item $[P,\,Q]:=P\circ Q-(-1)^{|P||Q|}Q\circ P$, 
\item $\daleth P:=[d,\,P]$. 
\end{itemize}

\vspace{0.3cm}

Let us recall the lemma (cf. e.g. \cite{PdBAT}\,),  

\begin{lemma} Let $R\in Aut(TM)$. Then, 
$$
d_{R}= d+[\tau(S),\,d]-r(R), 
$$

where $\tau$ is defined by (\ref{tau}), 
$S:= R-\Id$, and 
$r(R)$\, is a zero order differential operator  quadratic in $S$  defined as follows:\\
${}$\hspace{1cm}$\bullet$\, $r(R)\equiv 0$\, on $\wedge^{0}(M)$\\
${}$\hspace{1cm}$\bullet$\, for $\alpha\in\wedge^{1}(M)$  
$r(R)(\alpha)(X,\,Y)):=\alpha(R^{-1}N_{R}(X,\,Y))$,\\
${}$\hspace{1cm}$\bullet$\, extend to general $\alpha$ as skew-symmetric derivation.
\end{lemma}
\begin{proof} 
It is enough to prove the lemma for $f\in\wedge^{0}(M)$\,and for $\alpha\in\wedge^{1}(M)$\,.
If $f\in\wedge^{0}M)$\,, we have:
$$d_{R}f(X)=df((\Id+S)X)=df(X)+([\tau(S),\,d])f(X)\,;$$
if $\alpha\in\wedge^{1}(M)$\, we have first:
\begin{equation*}
\begin{split}
([\tau(S),\,d]\alpha)(X,\,Y)=&\\
&SX\alpha(Y)-SX\alpha(X)+\\
&\alpha(S[X,\,Y]-[SX,\,Y]-[X,\,SY])\,;
\end{split}
\end{equation*}
then:
\begin{equation*}
\begin{split}
(d_{R}\alpha)(X,\,Y)=&(dR^{-1}\alpha)((\Id+S)X,\,(\Id+S)Y)=\\
&(I+S)X\alpha(Y)-(\Id+S)Y\alpha(X)-\alpha((\Id+S)^{-1}[(\Id+S)X,\,(\Id+S)Y])\\
&X\alpha(Y)+SX\alpha(Y)-Y\alpha(X)-SY\alpha(X)-\alpha((\Id+S)^{-1}[(\Id+S)X,\,(\Id+S)Y])\\
&d\alpha(X,\,Y)+\alpha([X,\,Y])+([\tau(S),\,d]\alpha)(X,\,Y)+\\
&-\alpha(S[X,\,Y]-[SX,\,Y]-[X,\,SY])-\alpha((\Id+S)^{-1}[(\Id+S)X,\,(\Id+S)Y])=\\
&d\alpha(X,\,Y)+([\tau(S),\,d]\alpha)(X,\,Y)+\\
&\alpha((\Id+S)^{-1}((\Id+S)[X,\,Y]-(\Id+S)S[X,\,Y]+(\Id+S)[SX,\,Y]+(\Id+S)[X,\,SY])\\
&-\alpha((\Id+S)^{-1}[(\Id+S)X,\,(\Id+S)Y])=\\
&d\alpha(X,\,Y)+([\tau(S),\,d]\alpha)(X,\,Y)-\alpha((\Id+S)^{-1}N_{S}(X,\,Y)
\end{split}
\end{equation*}
\end{proof}

\n  We  will need the following 
\begin{lemma} \label{nijen} 
$$[d_{R},\,d]=0\,\,\,\Longleftrightarrow\,\,\,N_{R}=0 \, (\,\,\Longleftrightarrow\,\,\,d_{R}=d+[\tau(S),\,d]).$$
\end{lemma}

\begin{proof} { 
Let us first show that $d$ commutes with $d+[\tau(S),\,d]$. Since $d^2=0$, $d$ commutes with itself. 
In order to show that  $d$ commutes with  $[\tau(S),\,d]=0,$ we use the Jacobi identity:  $$[d,\,[\tau(S),\,d]]=[[d,\,\tau(S)],\,d]=-[d,\,[\tau(S),\,d]]\ \Longrightarrow \ [d, [\tau(S),\,d]]=0.$$  

(In particular,  the above observation shows the ``$\Longleftarrow$" direction  of the lemma)} 

In order to show that $d$ commutes with $r(R)$ if and only if $N_R=0$, 
we use that,     for every  $f\in\wedge^{0}(M)\,\,,\,\,X,\,Y\in TM$\,, we have,
$$[d,\,d_{R}]f(X,\,Y)=(r(R)df)(X,\,Y)=df(R^{-1}N_{R}(X,\,Y)). $$
Clearly, the right hand side vanishes for all $X,Y$ if and only if $N_R\equiv 0$. \end{proof}

\begin{rem} The previous lemma says that, in $(End(\wedge^{*}(M)),\,[\,,\,],\daleth)$\,,
$$\daleth d_{R}=0\,\,\,\Longleftrightarrow\,\,\,N_{R}=0\,\,\,\Longleftrightarrow\,\,\,d-d_{R}=\daleth\tau(S)\,\,{\rm i.e.}\,\,<d>=<d_{R}>\,.$$ 

Note also that:
$$[d,\,d_{R}]=0$$
$$\Updownarrow$$
$$d_{R}\,\,{\rm satisfies\,\,the\,\, Maurer-Cartan\,\,equation}\,\,\daleth d_{R}+\frac{1}{2}[d_{R},\,d_{R}]=0.$$\end{rem}

\section{$dd_{R}$-lemma in the presence of a Riemannian metric} \label{riem} 

\n Let $g$ be a Riemannian metric on $M$. We denote by $\ast$ the Hodge-star operation. The next two lemmas says that when certain (natural) conditions   on  $R$  are fulfilled, then the $dd_R-$lemma holds.    
\begin{lemma}
Let $R\in { Aut}(TM)$\, such that:
\begin{enumerate}
\item $N_{R}=0$\,( $\stackrel{\textrm{\tiny Lem. \ref{nijen}}}{\Longleftrightarrow}$  $d_{R}=d+[\tau(S),\,d]$\, { $\stackrel{\textrm{\tiny Lem. \ref{nijen}}}{\Longleftrightarrow}$} $[d,\,d_{R}]=0$\,)
\item there exists a Riemannian metric $g$\, on $M$\, such that
\begin{enumerate}
\item[a.] $[d_{R},\,dd^{*}]=0$
\item[b.] $[d_{R},\,d^{*}d]=0$\,. 
\end{enumerate}
\end{enumerate}
Then,   
$$Ker\,d\cap Im\,d_{R}=Im\,dd_{R}$$
\end{lemma}

\begin{proof} 
Set
$$\Delta_{R}:=[d_{R},\,d_{R}^{*}]\,.$$
Clearly, 
$$[\Delta_{R},\,d]=0=[\Delta,\,d_{R}]\,.$$

Note that $\Delta_{R}=R{\tilde{\Delta}}R^{-1}$\,, where ${\tilde{\Delta}}$\, is the Laplacian operator with respect to ${\tilde{g}}=g(R\cdot,\,R\cdot)$\,. 
\\
consider the Hodge decomposition with respect to $\Delta$\, and $\Delta_{R}$\,:
$$I=H+\Delta G$$
$$I=H_{R}+\Delta_{R}G_{R}\,.$$

Given $\alpha\in Ker\,d\cap Im\,d_{R}$\, we have:
$$\alpha=H(\alpha)+dd^{*}G(\alpha)$$
$$\alpha=d_{R}d_{R}^{*}G_{R}(\alpha)\,.$$
Set $\gamma=d_{R}^{*}G_{R}(\alpha)$\,.
Then, 
$$\gamma=H(\gamma)+dd^{*}G(\gamma)+d^{*}dG(\gamma)$$
and so
$$d_{R}\gamma=d_{R}H(\gamma)+dd^{*}G(d_{R}\gamma)\,,$$
i.e., 
$H(d_{R}\gamma)=$
$$d_{R}H(\gamma)=(d+[\tau(S),\,d])H(\gamma)=-d\tau(r)H(\gamma)=0$$
and so:
$$\alpha=dd^{*}G(\alpha)=d_{R}d_{R}^{*}G_{R}(\alpha).$$
and finally\\
$\alpha=dd^{*}G(\alpha)=d_{R}d_{R}^{*}G_{R}(dd^{*}G(\alpha))=dd_{R}G_{R}(d_{R}^{*}d^{*}G(\alpha))\,, $
\end{proof}

\begin{cor}\label{C:ddl}

Let $R\in End(TM)$\, such that:
\begin{enumerate}
\item $N_{R}=0$\,(and so $d_{R}=d+[\tau(S),\,d]$\, and $[d,\,d_{R}]=0$\,)
\item there exists a Riemannian metric $g$\, on $M$\, such that
\begin{enumerate}
\item[a.] $[d_{R},\,dd^{*}]=0$
\item[b.] $[d_{R},\,d^{*}d]=0$
\item[c.] $[d,\,d_{R}d_{R}^{*}]=0$
\item[d.] $[d,\,d_{R}^{*}d_{R}]=0$\,;
\end{enumerate}
\end{enumerate}
then the $dd_{R}$-lemma holds, i.e.
$$(Ker\,d\cap Ker\,d_{R})\cap(Im\,d+Im\,d_{R})=Im\,dd_{R}$$
\end{cor}

\section{$dd_R-$lemma in the  almost symplectic setting} \label{scew} 
\n Let $(M,\kappa)$\,be an almost symplectic, $2n$-dimensional compact manifold. We consider 
$${\mathcal{M}}_{\kappa}(M):=\left\{g\in{\mathcal{R}}iem(M)\,|\,d\mu(g)=\frac{\kappa^{n}}{n!}\right\}\,.$$

Recall (cf. \cite{PdB1})\,that we can define the symplectic analog of the Hodge-star  
$$\bigstar\,:\,\wedge^{r}(M)\,\lr\,\wedge^{2n-r}(M)$$
by means of the relation
$$\alpha\wedge\bigstar\beta=\kappa(\alpha,\,\beta)\frac{\kappa^{n}}{n!}$$
for $\alpha,\,\beta\in\wedge^{r}(M)$. 

Analog to the Riemannian case, we  consider on $\wedge^{r}(M)$\,:
$$d^{\bigstar}:=(-1)^{r}\bigstar d\bigstar\,.$$

\n For any 
$$g\in{\mathcal{M}}_{\kappa}(M)$$
there exists $R\in End(TM)$\,such that
$$g(X,\,Y)=\kappa(R^{-1}X,\,Y)\,.$$
Clearly,  $R$\,and $R^{-1}$\, are $g$-antisymmetric and $det\,R\equiv 1$\,;
\\
on $\wedge^{r}(M)$\, we have (cf. \cite{PdB2}\,):

$$\bigstar R^{-1}=(-1)^{r}*\,\,\,{\rm i.e}\,\,*R=(-1)^{r}\bigstar\,\,{\rm and}\,\,R^{-1}*=\bigstar\,;$$
consequently:
$$d_{R}^{*}=-R^{-1}*d*R=-d^{\bigstar}.$$

\n We have the following 
\begin{lemma}

Assume
\begin{itemize}
\item $N_{R}\equiv 0$
\item $d_{R}[d,\,d^{\bigstar}]=0=[d,\,d^{\bigstar}]d_{R}$
\item $d_{R^{-1}}[d,\,d^{\bigstar}]=0=[d,\,d^{\bigstar}]d_{R^{-1}}$
\end{itemize}

\n then, the $dd_{R}$-lemma holds.
\end{lemma}
\begin{proof}
We have:
$$[d,\,d_{R}d_{R}^{*}]=dd_{R}d_{R}^{*}-d_{R}d_{R}^{*}d=-d_{R}[d,\,d_{R}^{*}]=d_{R}[d,\,d^{\bigstar}]$$
and, similarly:
$$[d,\,d_{R}^{*}d_{R}]=-[d,\,d^{\bigstar}]d_{R}\,;$$
finally, we have:
$$[d_{R},\,dd^{*}]=-d[d^{*},d_{R}]$$
and
$$R^{-1}d[d^{*},\,d_{R}]R=\pm d_{R^{-1}}[d^{\bigstar},\,d].$$ 
Repeating this procedure with  the other relation and  applying  Lemma \ref{C:ddl} we obtain  that $dd_R-$lemma holds.  \end{proof}

\begin{rem}  If $\kappa$\, defines a symplectic structure, i.e. $d\kappa=0$\,, then $[d,\, d^{\bigstar}]=0$\, (cf. e.g. \cite{PdB2}\,),
 and so we only need $N_{R}\equiv 0$\.. \end{rem}

\section{Relaxing the condition $J^2=-\Id$ in the  definition of K\"ahler manifold.}
One of the equivalent definitions of the Kähler manifold is the following one: A Kähler manifold is 
a symplectic manifold $(M, \kappa)$ equipped with  $J\in End(TM)$ such that 
 the bilinear form  $g$\, defined by the  equality 
$g(X,\,Y):=\kappa(X,\,JY)$ is a Riemannian metric and such that \\
${I}$\hspace{1cm}\weg{$\bullet$}\, $J^{2}=-\Id$\\
${II}$\hspace{1cm}\weg{$\bullet$}\, $N_{J}=0$. 

 A lot of papers study the consequences of   
 relaxing the second condition $N_{J}=0$. In this case, the structure $J$ is called an {\it almost complex structure}, and many papers are dedicated to almost complex structures satisfying 
  additional conditions,   see for example \cite{Grey}.

 What about relaxing the first condition?

\begin{thm} \label{5.1} 
Let $(M,\kappa)$\,be an almost symplectic, $2n$-dimensional connected  manifold;\\
let again 
\begin{equation} \label{*} {\mathcal{M}}_{\kappa}(M):=\left\{g\in{\mathcal{R}}iem(M)\,|\,d\mu(g)=\frac{\kappa^{n}}{n!}\right\}\,.\end{equation}
Assume there exists $g\in{\mathcal{M}}_{\kappa}(M)$\, such that, representing $g$\,via $\kappa$ by $R\in End(TM)$\,, i.e. for $R$ satisfying 
$$g(X,\,Y)=\kappa(RX,\,Y)\,,$$
we have 
$$N_{R}\equiv 0\,.$$
Then, the orthogonal component $J$\, of $R$\, in its $g$-polar decomposition  is $g-$scew-symmetric and  satisfies
$$N_{J}\equiv 0\,.$$
Moreover, if $d\kappa=0$\,, then $(M,\,g,\,J)$\, is a K\"ahler manifold.
\end{thm}

\begin{proof} 
The proof is organized as follows: we will first show that  the  orthogonal component $J$\, of $R$\, in its $g$-polar decomposition is actually a polynomial of $R$ (we will also see that the polynomial is real and odd). The property $N_J=0$ will then follow from $N_R=0$ by \cite{BM}. The closedness of the form $g(J \cdot ,\cdot) $ will require certain additional work.

We consider $-R^2:= - R\circ R$. It is clearly 
 self adjoint and positively definite 
   with respect to $g$; by \eqref{*} we have $ \det(R^2)= \textrm{const}$. 
   Then, it is semi-simple, and all its eigenvalues are positive by linear algebra.
  
  We denote by  $m(x)$ the number of different eigenvalues of $-R^2$ at $x\in M$  and 
   by $\lambda_{1}(x)^2> ...>  \lambda_{m(x)}(x)^2\,\,\,(\lambda_{j}>0\,\, ,\,\,1\leq j\leq m(x))$ the   eigenvalues of $-R^2$ at $x\in M$. 
  
   We say that a point $x\in M$\, is  {\em stable}\,\,if $m(x)$\, is constant in a neighborhood of $x$\,. By \cite[Lemma 4]{MT},  the set of stable points is open and  everywhere dense on $M$.   Later, we will even show that all points are stable. 
We shall first  work near a stable point $x$.

    By \cite[Lemma 6]{BM}, the Nijenhuis tensor $N_{-R^2}=0$.  By \cite{H}, in the neighborhood of $x$ there exists a coordinate system  $\bar x= \left(\bar x_1= (x_1^1,,...,x_1^{2k_1}),..., \bar x_m =(x_m^1,...,x_m^{2k_m})\right) $ such that in this coordinate system the matrix of $-R^2$ is block diagonal, the dimensions of the blocks are $2k_1,...,2k_m$, and such that the $jth$ block is $\lambda_j^2$ times the identity $ 2k_j\times 2k_j$-matrix:

   \begin{equation}\label{L}
-R^2= \begin{pmatrix} \lambda^2_1 \cdot \Id_{2k_1}& & \\& \ddots & \\ &&\lambda^2_m\cdot \Id_{2k_m}  \end{pmatrix}.\end{equation}
Moreover, the function $\lambda_j$ does not depend on  the variables $x_i^\ell$ for $j\ne i$.

This in particular implies that all eigenvalues $\lambda_i$ are actually constant: indeed, from \eqref{L} we know that 
 the determinant 
of $-R^2$ is the product  $(\lambda_1)^{2k_1}\cdot ...\cdot (\lambda_m)^{2k_m}$. By assumption, the determinat is constant. Since the functions $\lambda_i$ depend on its own variables, all functions $\lambda_i$ must be constant. Then, all points  must be stable as we claimed before.

\begin{rem} For further use let us note that,  since the eigenspaces of $R$ corresponding to different eigenvalues are orthogonal, in  
 these coordinates the matrix of $g $ is also block-diagonal with the same as in \eqref{L} dimensions of the blocks; by construction, the components of $R$ are also orthogonal with the same   dimensions of the blocks  \begin{equation} \label{g} 
g= \begin{pmatrix} g_1 & & \\& \ddots & \\ &&g_m  \end{pmatrix}\ , \ \  R= \begin{pmatrix}R_1 &  \\& \ddots & \\ &&R_m  \end{pmatrix}
 \  .\end{equation}
\end{rem}  

 Let us now cook with the help of $R$ the field of endomorphisms $J$ such that it is the orthogonal component  $R$  in its $g-$polar decomposition.  We take the   polynomial $P(X) = a_{2m-1} X^{2m-1}+...+a_0$  of degree at most $2m-1$ 
  such that its value at the points    $X=\cplxi \lambda_1,  ... , \cplxi \lambda_m$ is equal to $\cplxi $ and such that  its value at the points    $X=-\cplxi \lambda_1,  ... , -\cplxi \lambda_m$ is equal to $-\cplxi $.  From general theory  it follows that such polynomial is unique (since the values in $2m$ points determine a unique  polynomial of degree $2m-1$, see  \cite[\S 2 Ch. 1]{Bach}). Since $P(\overline X)= \overline{P(X)}$ for $2m$ points $  X= \pm\cplxi \lambda_1,  ... , \pm\cplxi \lambda_m $, the coefficients of the 
   polynomial are real. Since $P(- X)= - {P(X)}$   for $2m$ points $  X= \pm\cplxi \lambda_1,  ... , \pm\cplxi \lambda_m $, the polynomial is odd (i.e., all terms of even degree are zero). 
   
 We would like to point out that,  since $ \lambda_i$ are constant, the coefficients of the polynomial are constant.

  We now consider $J:= P(R)= a_{2m-1} R^{2m-1}+...+ a_1\cdot R$ (we understand $R^{r}$ as 
  $\underbrace{R\circ R\circ  ....\circ R}_{\textrm{\tiny $r$ times}}$). Let us show that $J$ is indeed 
  the orthogonal component  of $R$ in its $g-$polar decomposition.

  Evidently, the eigenvalues of $J$ are $P(\pm  \cplxi \lambda_i)=  \pm \cplxi  $, and the algebraic multiplicity of each eigenvalue coincides with its geometric  multiplicity. Then, $J^2 =-{\Id}$. 
   
  Now,  since the polynomial $P$ is even, the bilinear form  $g(J\cdot, \cdot) $ is scew-symmetric. Indeed,  all terms of the polynomial  of even degree are zero, and 
  for every term 
  of odd degree we have  
  $$g(a_{2\ell -1}R^{2\ell-1}(U), V) =    -g(a_{2\ell -1}R^{2\ell-2} (U), R(V))=  
  g(a_{2\ell -1}R^{2\ell-3} (U), R^2(V))= ...=-g(U,a_{2\ell -1}R^{2\ell-1}(V) ) $$
  (each  time  we transport one $R$ to the right hand side  we change the sign; all together we  make odd number  the sign change).       
Then, each term $g(a_{2\ell -1}R^{2\ell-1}\cdot , \cdot )$ is skew-symmetric implying $g(J\cdot, \cdot) $ is scew-symmetric as well. 

Then, 
   $J$ is a $g-$orthogonal operator. Indeed, 
   $$g(JV,\underbrace{JU}_{U'} )=-g(J \underbrace{J U}_{U'}, V)  = g(U,V)= g(V,U).$$

Now, the operator $R\cdot J= R\cdot P(R)$ is $g-$symmetric (impyling $R= SJ$ for a certain  $g-$symmetric operator $S$). Indeed, arguing as above, we have 
   $$g(a_{2\ell -1}R^{2\ell}(U), V) =    -g(a_{2\ell -1}R^{2\ell-1} (U), R(V))=  
  g(a_{2\ell -1}R^{2\ell-3} (U), R^2(V))= ...=g(U,a_{2\ell -1}R^{2\ell}(V) ) $$
   (this time we transport $2\ell$ $R$'s from left to right, so we change the sign even number of times).
    Finally, $J=P(R)$ satisfies the following  properties:
    \begin{itemize}\item  It is $g-$orthogonal, 
\item $R= SJ$ for a certain  $g-$symmetric operator.
    \end{itemize}
    Thus, $J$ 
     is    the orthogonal component  of $R$ in its $g-$polar decomposition.
  
  Our goal is to show that $(g,J)$ is a Kähler structure on $M$ provided $\kappa$ s closed. We already have seen that $J$ is $g-$skew-symmetric.  The property  $N_J\equiv 0$ follows from \cite[Lemma 6]{BM}. 
   
Let us now  prove prove that the  form  $g(J\cdot , \cdot) $ is also closed. 
We will work locally, in a coordinate system $\bar x$ constructed above.
Combining these  with  the form  \eqref{L} of $-R^2$,
 we obtain that the matrix of  $J$ is given by  
 \begin{equation}\label{J}
J= \begin{pmatrix} \tfrac{1}{\lambda_1}R_1 & & \\& \ddots & \\ &&\tfrac{1}{\lambda_m}R_m  \end{pmatrix}\end{equation}

 Combining \eqref{L}
and \eqref{g} we see that the matrix of $\kappa(\cdot, \cdot):= g(R\cdot , \cdot )$  (in the coordinate system $ \bar x$ above)  is  given by the matrix 

\begin{equation}\label{omega}
\kappa= \begin{pmatrix} \om{1} & & \\& \ddots & \\ &&\om{m}  \end{pmatrix}= \begin{pmatrix} -R_1g_1 & & \\& \ddots & \\ &&-R_mg_m  \end{pmatrix}.\end{equation}

Then, by \eqref{J}, the matrix of $g(J\cdot , \cdot )$ is 
\begin{equation}\label{Omega}
 \begin{pmatrix} -Jg_1 & & \\ & \ddots & \\ &&-Jg_m  \end{pmatrix} = \begin{pmatrix} -\tfrac{1}{\lambda_1}R_1g_1 & & \\& \ddots & \\ &&-\tfrac{1}{\lambda_m}R_mg_m \end{pmatrix}  = \begin{pmatrix} \tfrac{1}{\lambda_1}\om{1} & & \\& \ddots & \\ &&\tfrac{1}{\lambda_m}\om{m}  \end{pmatrix}.\end{equation}

  In what follows we will us the convention $$\bar x= \left(\bar x_1= (x_1^1,,...,x_1^{2k_1}),..., \bar x_m =(x_m^1,...,x_m^{2k_m})\right)=\left(y^1,...,y^{2n}\right),   $$ 
  i.e., $y^1:=x_1^1 ,...,y^{2k_1}:=x_1^{2k_1}, y^{2k_1+1}:=x_2^1,..., y^{2n}:=x_m^{2k_m}.$

  Now we use that the differential of the form  is given by 
  $$d\left(\sum_{p,q=1}^{2n}{\kappa_{pq}} dy^{p}\wedge  dy^{q}\right)= \sum_{p,q,s=1}^{2n} \left(\tfrac{\partial}{\partial y^s} \kappa_{pq}\right) dy^s\wedge   dy^{p}\wedge  dy^{q}. $$ If the matrix of the 
  form $\kappa $ is as in \eqref{omega}, i.e., if 
  
 $$\kappa = \underbrace{\sum_{\alpha, \beta =1}^{2k_1}{ \om{1}_{\alpha \beta}}dx_1^{\alpha }\wedge  dx_1^{\beta}}_{\om{1}} + ...+ \underbrace{\sum_{\alpha, \beta =1}^{2k_m}{\om{m}_{\alpha \beta}}dx_m^{\alpha }\wedge  dx_m^{\beta}}_{\om{m}} , $$
  then, the  differential of $\kappa$ is 
  $$
  d\kappa = d\om{1} + ...+ d\om{m}= 
     \sum_{i=1}^m \left(\underbrace{\sum_{p=1}^{2n}\sum_{\alpha, \beta=1}^{2k_i} \left( \tfrac{\partial }{\partial y^p}\om{i}_{\alpha\beta}\right)  dy^{p}\wedge dx_i^{\alpha }\wedge  dx_i^{\beta}}_{d\om{i}}\right). 
  $$
  We see that the  components of the differentials of $d\om{i}$ and $d\om{j}$ do not combine for $i\ne j$. Indeed, every component of $d\om{i}$ is proportional to a certain $dy^{p}\wedge dx_i^{\alpha }\wedge  dx_i^{\beta}$,  and every component of $d\om{j}$ is proportional to a certain $dy^{p}\wedge dx_j^{\alpha }\wedge  dx_j^{\beta}$. Then, $d\kappa =0 $ implies $d\om{i}=0$ for all $i$.

  Now, by \eqref{Omega}, the form $g(J\cdot , \cdot )$ is given by 
  \begin{equation} \label{om1} 
  \tfrac{1}{\lambda_1}\om{1} + ...+ \tfrac{1}{\lambda_1}\om{m}. 
 \end{equation} 
 Since $\lambda_i$ are constants as we explained above,  and $d\om{i}=0$,  then 
  the differential of \eqref{om1} vanishes. Thus,   $g(J\cdot , \cdot )$ is closed as we claim. Theorem \ref{5.1} is proved.  \end{proof}

\begin{definition}
Let $(M\,\kappa)$\, be a $2n$-dimensional (compact) symplectic manifold;\\
$R\in End(TM)$\,is said to be $\kappa$-{\em calibrated}\, if
$$g:=\kappa(R\cdot,\,\cdot)$$
is a Riemannian metric such that $d\mu(g)=\frac{\kappa^{n}}{n!}$\,.
\end{definition}

From Theorem \ref{5.1} we immediately obtain 
\begin{cor}
Let $(M\, , \kappa)$\, be a $2n$-dimensional { connected}  symplectic manifold. \\
Then the following statements are equivalent:
\begin{itemize}
\item $(M\, ,  \kappa)$  admits a Kähler structure $g, J$ such that 
$\kappa(\cdot, \cdot) =g(J \cdot, \cdot)$. 
 \item there exists $R\in End(TM)$\, such that it is  $\kappa$-calibrated and such that $N_{R}\equiv 0$.
\end{itemize}
\end{cor}

 \subsubsection*{Acknowledgements} We thank A. Ghigi, G. Manno and F. Magri for  useful  discussions. The work was start at the   Milano Bicocca University; we thank it  for the hospitality. 
The stay of  V.M. at Milano Bicocca was supported by   GNSAGA. V.M. also thanks Deutsche Forschungsgemeinschaft
(Priority Program 1154 --- Global Differential Geometry and Research Training Group   1523 --- Quantum and Gravitational Fields),
  and FSU Jena for partial financial support.

\n Vladimir S. Matveev: Institute of Mathematics, FSU Jena, 07737 Jena Germany, \\  vladimir.matveev@uni-jena.de
\end{document}